\documentclass[11pt]{amsart}
\usepackage[utf8]{inputenc}
\usepackage[T1]{fontenc}

\usepackage{amssymb,amsmath,amsthm,newlfont,enumerate}
\usepackage{array}

\newtheorem{theorem}{Theorem}[section]
\newtheorem*{theorem*}{Theorem}
\newtheorem{lemma}[theorem]{Lemma}
\theoremstyle{remark}
\newtheorem*{remark*}{Remark}
\theoremstyle{definition}
\newtheorem*{example}{Example}
\usepackage{graphicx}
\usepackage{float}              % Pour figures où on veut (mettre [h] )
\usepackage{subcaption}         % Pour def environnement subfigure

\begin{document}

\title{Pseudospectra and Simultaneous Power Control.}
\author{Ma\"eva Ostermann}
\address{D\'epartement de math\'ematiques et de statistique, Universit\'e Laval,
Qu\'ebec City (Qu\'ebec),  Canada G1V 0A6.}
\email{maeva.ostermann.1@ulaval.ca}
\date{}
\begin{abstract}
We prove that, for $M>0$ and $n,m\ge2$, we can find two $10\times10$ matrices $A$ and $B$ with identical pseudospectra such that we have simultaneously  $\|A^n\|/\|B^n\|>M$ and  $\|A^m\|/\|B^m\|>M$. We also prove that, under certain conditions, this result holds for two more general functions of $A$ and $B$.
\end{abstract}

\subjclass[2010]{primary 47A10, secondary 15A18, 15A60}

\keywords{matrix, norm, pseudospectra}
\maketitle

\section{Introduction}
Given $N\ge1$, let $M_N(\mathbb C)$ be the algebra of complex $N\times N$ matrices.
Let $\|.\|_2$ be the Euclidean norm on $\mathbb C^N$ and let $\|.\|$ be the associated operator norm on $M_N(\mathbb C)$.

In many fields of mathematics, the study of the powers of a matrix and their asymptotic behavior plays an essential role. For this, it can be useful to study the spectrum. The spectral radius of a matrix $A$ allows for example to govern in the long term the norm of the powers of $A$. In fact, by the spectral radius formula, we have
\[\lim_{n\to\infty}\|A^n\|^{1/n}=\rho(A).\]
However it does not tell the whole story, especially for non-normal matrices. As we can see with the example in section 3 in Trefethen and Embree's book \cite{TrefEmbr2005}, although the powers tend towards 0, they can still become enormous. This example is the $N\times N$ matrix defined by
\[A:=\begin{pmatrix}
0&1&&&\\
\frac14&0&1&&\\
&\ddots&\ddots&\ddots&\\
&&\frac14&0&1\\
&&&\frac14&0
\end{pmatrix}\]
The eigenvalues of $A$ are $\cos\frac{k\pi}{N+1}$, for $k=1\dots N$, so its spectral radius is $\rho(A)=\cos\frac\pi{N+1}<1$ and thus $\|A^n\|\to0$. However, as illustrated in \cite{TrefEmbr2005}, the norm of the powers of $ A $ are getting very large. The table below gives some values of $\|A^n\|$ when $N=64$.
\[\text{ \begin{tabular}{|c|c|c|c|c|}\hline
    $n$         &    $50$      & $500$                & $5000$               & $50000$  \\ \hline
    $\|A^n\|$   & $64171.071$  & $4.281\times10^{14}$ & $5.624\times10^{12}$  & $8.216\times10^{-11}$\\
\hline
 \end{tabular}}\]
It is thus of interest to have better control over the powers of a matrix, not just in the long term but also the $\sup\|A^n\|$. The Kreiss matrix theorem is a result that gives a good prediction of $\sup\|A^n\|$. Let us recall this theorem.
\begin{theorem*}[Kreiss Matrix Theorem \cite{WegTref1994}]
Let $N\in\mathbb N$ and $A\in M_N(\mathbb C)$, a $N\times N$ matrix. Let \[r(A):=\sup\limits_{|z|>1}(|z|-1)\|(zI-A)^{-1}\|.\]
Then
\[r(A)\le \sup_{k\ge0}\|A^k\|\le eN r(A).\]
\end{theorem*}
This prediction is based on the resolvent of the matrix. This brings us to the notion of pseudospectra.
Given a $N\times N$ matrix $A\in M_N(\mathbb C)$ and $\varepsilon>0$, the $\varepsilon$-pseudospectrum of $A$ is defined by
\[\sigma_\varepsilon(A):=\left\{z\in\mathbb C,\; \|(zI-A)^{-1}\|> 1/\varepsilon\right\},\]
with the convention $\|(zI-A)^{-1}\|=\infty$ if $z\in\sigma(A)$, the spectrum of $A$. 
For more information about pseudospectra and their applications, see the book of Trefethen and Embree \cite{TrefEmbr2005}.

 In this paper, we are interested of matrices with the same $\varepsilon$-pseudospectrum, for all $\varepsilon$ and more precisely, to the powers taken individually of matrices having identical pseudospectra. We say that two matrices $A$ and $B$ have identical pseudospectra if, for all $\varepsilon>0$,  $\sigma_\varepsilon(A)=\sigma_\varepsilon(B)$, i.e. if \[\|(zI-A)^{-1}\|=\|(zI-B)^{-1}\|\qquad(z\in\mathbb C).\]

In 1993, Greenbaum and Trefethen posed the question: ``Do pseudospectra determine the behavior of a matrix?'' They gave a negative answer to their question by producing a example of two matrices with identical pseudospectra but with different norm (see \cite{GreenTref1993} or \cite[\textsection~47]{TrefEmbr2005}). However, the pseudospectra more or less determine the norm: it is known that, if $A,B\in M_N(\mathbb C)$ have  identical pseudospectra, then $\|A\| \le 2\|B\|$ (see \cite{Rans2010}).
But, if we consider some powers higher than $1$, then the story is really different. Ransford proved in \cite{Rans2007} that, for all $M>0$, there exist $A,B\in M_N(\mathbb C)$ such that $A$ and $B$ have identical pseudospectra and
\[\|A^n\|>M \|B^n\|\qquad(2\le n\le (N-3)/2).\]
In his paper, Ransford constructed two matrices with an excellent control of powers until $(N-3)/2$ but  $A^n=B^n=0$ for $n>(N-3)/2$. These matrices are nilpotent, so this does not tell us if the pseudospectra determine or not the asymptotic behavior of powers of matrices. This question remained open until 2013. Ransford and Raouafi in \cite{RansRaoua2013} showed that if $n\ge 2$ and $M>0$, then there exist $A,B\in M_6(\mathbb C)$ such that $A$ and $B$ have  identical pseudospectra and \[\|A^n\|>M \|B^n\|.\]
In fact, they proved the following more general result: if $f$ is an holomorphic function in a domain $\Omega\subset\mathbb C$, that is not a M\"obius transformation, and if $M>0$, then there exist $A,B\in M_6(\mathbb C)$ with their spectrum in $\Omega$ such that $A$ and $B$ have  identical pseudospectra and \[\|f(A)\|>M\|f(B)\|.\]
The construction of these matrices $A$ and $B$ depends very strongly on the power (or the function) considered. It is therefore not clear that one can obtain a similar result for a sequence of powers, starting with two powers. Our first theorem solves this problem in the case of two powers.
\begin{theorem}\label{TheoP}
Let $m, n\in\mathbb N$ with $n,m\ge2$. Let $M>0$. Then there exist $A,B\in M_{10}(\mathbb C)$ such that $A$ and $B$ have identical pseudospectra and
\[\frac{\|A^n\|}{\|B^n\|}>M \quad\text{and}\quad \frac{\|A^m\|}{\|B^m\|}>M.\]
\end{theorem}
We will prove this result in Section \ref{SecP}. In the Section \ref{SecF}, we will notice that, by adapting the construction, we can obtain the following result for two holomorphic functions, subject to a technical condition.
\begin{theorem}\label{TheoF}
Let $f$ and $g$ be two holomorphic functions in a domain  $\Omega\subset \mathbb C$ and let $M>0$ be a constant.
Suppose that $f$ and $g$  satisfy the condition
\[\tilde h_{1,2}(\tilde h_{2,4}^2-\tilde h_{3,4}(2\tilde h_{2,3}+\tilde h_{1,4}))+\tilde h_{1,3}(\tilde h_{1,3}\tilde h_{3,4}-\tilde h_{2,3}\tilde h_{2,4})+\tilde h_{2,3}^3\neq0,\qquad(\star)\]
where $\tilde h_{i,j}:=\frac1{i!j!}(f^{(i)}g^{(j)}-f^{(j)}g^{(i)})$.
Then there exist  $A,B\in M_{10}(\mathbb C)$ with spectrum in $\Omega$ such that $A$ and $B$ have identical pseudospectra and
\[\frac{\|f(A)\|}{\|f(B)\|}>M\qquad\text{and}\qquad\frac{\|g(A)\|}{\|g(B)\|}>M.\]
\end{theorem}
\section{Case of two powers.}\label{SecP}
In the proof of the Theorem \ref{TheoP}, we will construct two matrices $A_t$ and $B_t$ as a direct sum of different $2\times2$ niloptent matrices and a same  strictly upper triangular Toeplitz $8\times8$ matrix, which we will denote by $C_t$. We then verify that in this way, the matrices $A_t$ and $B_t$ have identical pseudospectra and that, for $t$ sufficiently large, the matrices definied by $A=I+A_t$ and $B=I+B_t$ verify $\|A^n\|/\|B^n\|>M$ and  $\|A^m\|/\|B^m\|>M$. The following lemma will be used to verify that the matrices thus constructed do indeed have identical pseudospectra.
\begin{lemma}\label{Lemma}% A reformuler
Let $S$ be the unilateral shift on $\mathbb C^8$. For $t\ge1$, we define \[C_t:=tS+c_3t^3S^3+c_4t^4S^4+c_5t^5S^5+c_6t^6S^6+c_7t^7S^7.\]
There exists a polynomial condition in $c_k$ (given in the proof) such that, if this condition is satisfied, then there exists $\mu>0$ such that
\[ \|(I-zC_t)^{-1}\|\ge 1+\mu t^6|z|\qquad(t\ge1,z\in\mathbb C).\]
\end{lemma}
This result is a direct consequence of the Lemma 2.1 in \cite{RansRaoua2013}. So let us start by recalling this lemma.
\begin{lemma}\label{L:RansRaou2013}
Let $S$ be the unilateral shift on $\mathbb C^N$. Let $\alpha_1,\dots,\alpha_{N-1}\in\mathbb C$. Then
\[1+\frac12\max_{1\le k\le N-1}|\alpha_k|\le\left\|I+\sum_{k=1}^{N-1}\alpha_kS^k \right\|\le1+\sum_{k=1}^{N-1}|\alpha_k|.\]
\end{lemma}
\begin{proof}[Proof of Lemma \ref{Lemma}] We will first show that there exist two polynomials $P_1(z)$ and $P_2(z)$ such that
\[\|(I-zC_t)^{-1}\|\ge 1+\frac{t^6|z|}{2}\max(|P_1(z)|,|P_2(z)|) \quad(t\ge1,z\in\mathbb C).\]

Since $S^k=0$ for $k\ge8$, then $C_t^k=0$ if $k\ge8$. So
when we compute $(I-zC_t)^{-1}$, we obtain
\begin{align*}
(I-zC_t)^{-1}
&=\sum_{k=0}^\infty (zC_t)^k=\sum_{k=0}^7 z^kC_t^k\\
& = I+ztS+z^2t^2S^2+(z^3+c_3z)t^3S^3\\
& +(z^4+2c_3z^2+c_4z)t^4S^4+(z^5+3c_3z^3+2c_4z^2+c_5z)t^5S^5\\
& +(z^6+4c_3z^4+3c_4z^3+(c_3^2+2c_5)z^2+c_6z)t^6S^6\\
&+(z^7+5c_3z^5+4c_4z^4+3(c_3^2+c_5)z^3+2(c_3c_4+c_6)z^2+c_7z)t^7S^7.
\end{align*}
Then,  when $t\ge1$, by applying the Lemma \ref{L:RansRaou2013}, we have
\begin{align*}
\|(I-zC_t)^{-1}\|&\ge 1+\frac12\max\Big(|tz|,|z^2t^2|,|(z^3+c_3z)t^3|,\\
&\quad|(z^4+2c_3z^2+c_4z)t^4|,|(z^5+3c_3z^3+2c_4z^2+c_5z)t^5|,\\
&\quad|(z^6+4c_3z^4+3c_4z^3+(c_3^2+2c_5)z^2+c_6z)t^6|,\\
&\quad|t^7(z^7+5c_3z^5+4c_4z^4+3(c_3^2+c_5)z^3+2(c_3c_4+c_6)z^2+c_7z)|\!\Big)\\
&\ge 1+\frac12\max\Big(|(z^6+4c_3z^4+3c_4z^3+(c_3^2+2c_5)z^2+c_6z)t^6|,\\
&\quad|t^7(z^7+5c_3z^5+4c_4z^4+3(u^2+c_5)z^3+2(c_3c_4+c_6)z^2+c_7z)|\!\Big)\\
&\ge 1+\frac{t^6|z|}{2}\max(|P_1(z)|,|P_2(z)|)
\end{align*}
with 
\[P_1(z):=c_7+2z(c_3c_4+c_6)+3z^2(c_3^2+c_5)+4z^3c_4+5c_3z^4+z^6\]
 and
\[ P_2(z):=c_6+z(c_3^2+2c_5)+3z^2c_4+4c_3z^3+z^5.\]
Let $\mu:=\frac12\inf\limits_{z\in\mathbb C}\max(|P_1(z)|,|P_2(z)|)$. Then
\[ \|(I-zC_t)^{-1}\|\ge 1+\mu t^6|z|\qquad(t\ge1,z\in\mathbb C)\]
and $\mu>0$ if and only if $gcd(P_1,P_2)=1$.\\
Let $P_3(z):=c_3z^4+c_4z^3+(2c_3^2+c_5)z^2+(2c_3c_4+c_6)z+c_7$ be the remainder of the Euclidean division of $P_1$ by $P_2$.
Suppose $c_3\neq0$ and let $P_4 = Az^3+Bz^2+Cz +D$ be the remainder of the Euclidean division of $c_3^2P_2$ by $P_3$. Thus
\begin{align*}
A&:=2c_3^3-c_5c_3+c_4^2
\\B&:=3c_4c_3^2-c_6c_3+c_4c_5
\\C&:=u^4+2c_5c_3^2+(2c_4^2-c_7)c_3+c_4c_6
\\D&:=c_6c_3^2+c_4c_7.
\end{align*}
Assume $A\neq0$ and let $P_5 = Ez^2 + Fz + G$, the remainder of the Euclidean division of $A^2P_3$ by $P_4$. Then
\begin{align*}
E&:=A^2(2c_3^2+c_5)-ACc_3+B^2c_3-ABc_4
\\F&:=A^2(2c_4c_3+c_6)+C(Bc_3-Ac_4)-ADc_3
\\G&:=D(Bc_3-Ac_4)+A^2c_7.
\end{align*}
Suppose $E\neq0$ and let $P_6 = Hz+I$, the remainder of the Euclidean division of $E^2P_4$ by $P_5$. Then
\begin{align*}
H&:=-AEG+AF^2-BEF+CE^2
\\I&:=(AF-BE)G+DE^2.
\end{align*}
Finally let $ J = EI^2 - FHI + GH^2$. To conclude, if $c_k$ saisfy the conditions $c_3\neq0$, $A\neq0$, $E\neq0$ and $J\neq0$, then $gcd(P_1,P_2)=1$ and there exists $\mu>0$ such that
\[ \|(I-zC_t)^{-1}\|\ge 1+\mu t^6|z|\qquad(t\ge1,z\in\mathbb C).\qedhere
\]
\end{proof}
\begin{proof}[Proof of Theorem \ref{TheoP}]
If $n=m$, then we can apply the result in \cite{RansRaoua2013}.\\ 
We now suppose that $n\neq m$. Let $f(z):=(1+z)^n$ and $g(z):=(1+z)^m$.
For $k\in\mathbb N$, let $f_k:=\binom{n}{k}$ if $k\le n$, $f_k:=0$ if $k>n$,  $g_k:=\binom{m}{k}$ if $k\le m$ and $g_k:=0$ if $k>m$. And for $i,j\in\mathbb N$, we define $h_{i,j}:=f_ig_j-f_jg_i$.\\
Let $S$ be the unilateral shift on $\mathbb C^8$. Let $C_t$ be defined by
$$C_t:=tS+ut^3S^3+c_4t^4S^4+c_5t^5S^5+c_6t^6S^6+c_7t^7S^7\in M_8(\mathbb C).$$
$\bullet$ We begin by finding the $c_k$, rational functions in $u$, such that, when $t\to\infty$, $\|(I+C_t)^{n}\|= O(t^5)$ and $\|(I+C_t)^{m}\|=O(t^5)$. Since $S^k=0$ for $k\ge8$, then $C_t^k=0$ if $k\ge8$. So $f(C_t)=(I+C_t)^{n}$ and $g(C_t)=(I+C_t)^{m}$ are determined by the Taylor developments of $f$ and $g$ of degree $7$ that are, when $z\to0$, 
\[f(z)=f_0+f_1z+f_2z^2+f_3z^3+f_4z^4+f_5z^5+f_6z^6+f_7z^7+O(z^8)\]
\[g(z)=g_0+g_1z+g_2z^2+g_3z^3+g_4z^4+g_5z^5+g_6z^6+g_7z^7+O(z^8).\]
Then, we have 
\begin{align*}
f(C_t)&=(I+C_t)^{n}\\
&=f_0I+f_1C_t+f_2C_t^2+f_3C_t^3+f_4C_t^4+f_5C_t^5+f_6C_t^6+f_7C_t^7\\
&=f_0I+f_1tS+f_2t^2S^2+(f_3+uf_1)t^3S^3\\
&\quad +(f_4+2uf_2+c_4f_1)t^4S^4+(f_5+3uf_3+2c_4f_2+c_5f_1)t^5S^5\\
&\quad +(f_6+4uf_4+3c_4f_3+(u^2+2c_5)f_2+c_6f_1)t^6S^6\\
&\quad +(f_7+5uf_5+4c_4f_4+3(u^2+c_5)f_3+2(uc_4+c_6)f_2+c_7f_1)t^7S^7
\intertext{and}
g(C_t)&=(I+C_t)^{m}\\
&=g_0I+g_1C_t+g_2C_t^2+g_3C_t^3+g_4C_t^4+g_5C_t^5+g_6C_t^6+g_7C_t^7\\
&=g_0I+g_1tS+g_2t^2S^2+(g_3+ug_1)t^3S^3\\
&\quad +(g_4+2ug_2+c_4g_1)t^4S^4+(g_5+3ug_3+2c_4g_2+c_5g_1)t^5S^5\\
&\quad +(g_6+4ug_4+3c_4g_3+(u^2+2c_5)g_2+c_6g_1)t^6S^6\\
&\quad +(g_7+5ug_5+4c_4g_4+3(u^2+c_5)g_3+2(uc_4+c_6)g_2+c_7g_1)t^7S^7.
\end{align*}
Thus the $c_k$ are the solutions of
\[\begin{cases}
f_1c_7+2f_2(uc_4+c_6)+3f_3(u^2+c_5)+4f_4c_4+5uf_5+f_7=0\\
f_1c_6+f_2(u^2+2c_5)+3f_3c_4+4uf_4+f_6=0\\
g_1c_7+2g_2(uc_4+c_6)+3g_3(u^2+c_5)+4g_4c_4+5ug_5+g_7=0\\
g_1c_6+g_2(u^2+2c_5)+3g_3c_4+4ug_4+g_6=0,
\end{cases}\]
and this is true if and only if the $c_k$ are solutions of
\[ \left\{
\begin{matrix}
f_1c_7 & +2f_2c_6 & +3f_3c_5 & +(4f_4+2uf_2)c_4 & = & -(3f_3u^2+5f_5u+f_7) \\
       &  f_1c_6  & +2f_2c_5 &    +3f_3c_4      & = & -(f_2u^2+4f_4u+f_6)  \\
g_1c_7 & +2g_2c_6 & +3g_3c_5 & +(4g_4+2ug_2)c_4 & = & -(3g_3u^2+5g_5u+g_7) \\
       &  g_1c_6  & +2g_2c_5 &    +3g_3c_4      & = & -(g_2u^2+4g_4u+g_6).  \\
\end{matrix}
\right.\]
If we define
\begin{align*}
U&= \det\begin{pmatrix}
f_1&2f_2&3f_3&4f_4+2uf_2\\
0&f_1&2f_2&3f_3\\
g_1&2g_2&3g_3&4g_4+2ug_2\\
0&g_1&2g_2&3g_3\\
\end{pmatrix},
\intertext{then}
U&=
(4f_1^2g_2^2-8f_1f_2g_1g_2+4f_2^2g_1^2)u\\
&\quad+f_1^2(8g_2g_4-9g_3^2)+f_1f_2(12g_2g_3-8g_1g_4)+f_3(f_1(18g_1g_3-12g_2^2)\\
&\quad +12f_2g_1g_2)-12f_2^2g_1g_3+f_4(8f_2g_1^2-8f_1g_1g_2)-9f_3^2g_1^2\\
&=m^2n^2(n-m)^2u-\frac{m^2n^2(n-m)^2((n-m)^2-1)}{12}.
\end{align*}
Therefore we can choose $u$ such that $U\neq0$ and then the system has a unique solution. Moreover the solutions $c_k$ satisfy, when $u\to\infty$,
\begin{align*}
%%%%%%%%%%%%%%%%%%%%%%%%%%%%%%%%%%%%%%
    b_4
&:=Uc_4=\frac{-m^2n^2(n-m)^2(n+m-3)}4u^2\\&\quad+\frac{m^2n^2(n-m-1)(n-m)^2(n-m+1)(n+m-4)}{24}u+O(1)\\
%%%%%%%%%%%%%%%%%%%%%%%%%%%%%%%%%%%%%
    b_5
&:=Uc_5=\frac{-m^2n^2(n-m)^2}2u^3+\frac{m^2n^2(n-m)^2(n+m-3)}4u^2+O(u)\\
%%%%%%%%%%%%%%%%%%%%%%%%%%%%%%%%%%%%%%
    b_6
&:=Uc_6=n^2m^2(n-1)(m-1)(n-m)^2\left[\frac{(n+m-11)}{24}u^2\right.\\&\quad +\left.\frac{(3n^3-2mn^2-12n^2-2m^2n+3mn+33n+3m^3-12m^2+33m-72)}{360}u\right]+O(1)\\
%%%%%%%%%%%%%%%%%%%%%%%%%%%%%%%%%%%%%%
b_7&:=Uc_7=n^2m^2(n-1)(m-1)(n-m)^2\left[\frac14u^3\right.\\& \hspace{3cm}\left.-\frac{n^2-mn-2n+m^2-2m-3}{24}u^2\right]+O(u).
\end{align*}
$\bullet$ Now, we  verify the conditions on the resolvent. When $u\to\infty$, we have
\begin{align*}
A'&:=U^2A=m^4n^4(n-m)^4\times \\&\quad\left[\frac52u^5\frac{5n^2-14mn+10n+5m^2+10m-27}{16}u^4\right]+O(u^3)\\
%%%%%%%%%%%%%%%%%%%%%%%%%%%%%%%%%%%%%%
%B'&:=U^2B=m^4n^4(n-m)^4\left[\frac{-5(n+m-3)}8u^5\right. +(4n^3-5mn^2-15n^2\\&\quad\left.-5m^2n+39mn-7n+4m^3-15m^2-7m+12)\frac{u^4}{24}+O(u^3)\right]\\
%%%%%%%%%%%%%%%%%%%%%%%%%%%%%%%%%%%%%%
%C'&:=U^2C=m^4n^4(n-m)^4\left[\frac{n^2+4mn+m^2-13}{24}u^5\right. -(74n^4+9mn^3-525n^3\\&\quad\left.-146m^2n^2+315mn^2+1115n^2+9m^3n+315m^2n-505mn-1155n+74m^4\right.\\&\qquad\left.-525m^3+1115m^2-1155m+971)\frac{u^4}{1440}+O(u^3)\right]\\
%%%%%%%%%%%%%%%%%%%%%%%%%%%%%%%%%%%%%%
%D'&:=U^2D=(m-1)m^4(n-1)n^4(n-m)^4\left[\frac{-(n+m+13)}{48}u^5+(37n^3-18mn^2-128n^2\right.\\&\quad\left.-18m^2n+7mn+167n+37m^3-128m^2+167m-148)\frac{u^4}{1440}+O(u^3)\right]\\
%%%%%%%%%%%%%%%%%%%%%%%%%%%%%%%%%%%%%%
E'&:=U^5E=m^{10}n^{10}(n-m)^{10}\times \\&\quad\left[\frac{75}8u^{13}-\frac{5(67n^2-158mn+60n+67m^2+60m-199)}{96}u^{12} \right]+O(u^{11})\\
%%%%%%%%%%%%%%%%%%%%%%%%%%%%%%%%%%%%%%
%F'&:=U^5F=m^{10}n^{10}(n-m)^{10}\left[ \frac{-25(n+m-3)}8u^{13}+5(25n^3-31mn^2-61n^2\right.\\&\quad\left.-31m^2n+212mn-133n+25m^3-61m^2-133m+241)\frac{u^{12}}{96}+O(u^{11})\right]\\
%%%%%%%%%%%%%%%%%%%%%%%%%%%%%%%%%%%%%%
%G'&:=U^5G=(m-1)m^{10}(n-1)n^{10}(n-m)^{10})\left[ \frac{25}{16}u^{13}-\right.\\&\quad\left.\frac{5(25n^2-52mn+10n+25m^2+10m-111)}{192}u^{12}+O(u^{11})\right]\\
%%%%%%%%%%%%%%%%%%%%%%%%%%%%%%%%%%%%%%
%H'&:=U^{12}H= m^{24}n^{24}(n-m)^{24}\left[\frac{625(n^2-mn+m^2-3)}{64}u^{31}\right.\\&\quad\left.-125(139n^4-372mn^3-120n^3+439m^2n^2+180mn^2-400n^2\right.\\&\qquad\left.-372m^3n+180m^2n+754mn+139m^4-120m^3-400m^2\right.\\&\qquad\quad\left.+93)\frac{u^{30}}{1536}+O(u^{29})\right]\\
%%%%%%%%%%%%%%%%%%%%%%%%%%%%%%%%%%%%%%
%I'&:=U^{12}I=(m-1)m^{24}(n-1)n^{24}(n-m)^{24}\left[\frac{-625(n+m+3)}{128}u^{31}\right.\\&\quad\left.+125(139n^3-140mn^2+251n^2-140m^2n-539mn-149n+139m^3\right.\\&\qquad\left.+251m^2-149m-213)\frac{u^{30}}{3072}+O(u^{29})\right]\\
%%%%%%%%%%%%%%%%%%%%%%%%%%%%%%%%%%%%%%
J'&:=U^{29}J=n^{58}m^{58}(m^2-1)(n^2-1)(n-m)^{58}\times \\&\quad\left[\frac{-29296875(n^2-4mn+m^2+3)}{131072}u^{75}+\frac{1953125}{
1572864}p(n,m)u^{74}\right]+O(u^{73})
\end{align*}
with 
\begin{multline*}p(n,m)=446n^4-2649mn^3-60n^3+4412m^2n^2+180mn^2+982n^2-2649m^3n\\ +180m^2n-1411mn-180n+446m^4-60m^3+982m^2-180m-828.
\end{multline*}
We need to see that, with our conditions on $n$ and $m$, we cannot have $n^2-4nm+m^2+3= 0$ and $p(n,m)=0$ at the same time. With this condition, we can also choose $u\neq0$ such that $A\neq0$, $E\neq0$ and $J\neq0$ and apply the Lemma \ref{Lemma}.

Assume the opposite. We have $$p(n,m)=q(n,m)(n^2-4mn+m^2+3)+r(n,m)$$
with $q(n,m)=446n^2+(-865m-60)n+506m^2-60m-356$ and 
$r(n,m)=60(m-1)(m+1)(4mn-m^2-4)$.
Then, we deduce that  $r(n,m)=0$.
Moreover, we have
\[n^2-4mn+m^2+3=\frac{4mn-15m^2+4}{16m^2}r(n,m)+\frac{m^4-8m^2+16}{16m^2}\]
Then we conclude that $m=\pm2$, but this implies that $n=1$ and we finish with a contradiction. Then $\deg(J')\ge74$.
Thus we can finally choose $u\neq0$ such that we also have $A\neq0$, $E\neq0$ and $J\neq0$.
By Lemma \ref{Lemma}, there exists $\mu>0$ such that 
\[\|(I-zC_t)^{-1}\|\ge 1+\mu t^6|z|\qquad(t\ge1,z\in\mathbb C).\]
$\bullet$ We finish the proof with the construction of $A$ and $B$.\\
Let  $A_t:= C_t\oplus \begin{pmatrix}
0&\mu t^6\\0&0
\end{pmatrix}\in M_{10}(\mathbb C)$ and $B_t:= C_t\oplus O_2\in M_{10}(\mathbb C)$.
We have, for all $t\ge 1$ and $z\in\mathbb C$,
\[\left\|\left(I-z\begin{pmatrix}
0&\mu t^6\\0&0
\end{pmatrix}\right)^{-1}\right\| =\left\|\begin{pmatrix}
1&z\mu t^6\\0&1
\end{pmatrix}\right\|\le 1+\mu t^6|z|.\]
Then
\[ \|(I-zA_t)^{-1}\|=\|(I-zC_t)^{-1}\|=\|(I-zB_t)^{-1}\|.\]
Now, if $t\ge1$, we see that 
\begin{align*}
\|(I+A_t)^n\|&\ge\left\|\left(I+\begin{pmatrix}
0&\mu t^6\\0&0
\end{pmatrix}\right)^n\right\|=\left\| \begin{pmatrix}
1&n\mu t^6\\0&1
\end{pmatrix}\right\|\ge n\mu t^6 \\
\|(I+B_t)^n\|&=\max(\|(I+C_t)^n\|, 1)=\|(I+C_t)^n\|=O(t^5)\quad(t\to\infty)\\
\|(I+A_t)^m\|&\ge\left\|\left(I+\begin{pmatrix}
0&\mu t^6\\0&0
\end{pmatrix}\right)^m\right\|=\left\|\begin{pmatrix}
1&m\mu t^6\\0&1
\end{pmatrix}\right\|\ge m\mu t^6 \\
\|(I+B_t)^m\|&=\max(\|(I+C_t)^m\|, 1)=\|(I+C_t)^m\|=O(t^5)\quad(t\to\infty).
\end{align*}
Since $ \mu>0$, we have
\[\lim_{t\mapsto\infty}\frac{\|(I+A_t)^n\|}{\|(I+B_t)^n\|}=+\infty\quad\text{and}\quad \lim_{t\mapsto\infty}\frac{\|(I+A_t)^m\|}{\|(I+B_t)^m\|}=+\infty.\]
So, finally, we can choose $t\ge1$ such that
\[\frac{\|(I+A_t)^n\|}{\|(I+B_t)^n\|}>M\quad\text{and}\quad \frac{\|(I+A_t)^m\|}{\|(I+B_t)^m\|}>M.\]
To conclude, if we take $A=I+A_t$ and $B=I+B_t$ then $A$ and $B$ have identical pseudospectra and 
\[\frac{\|A^n\|}{\|B^n\|}>M\quad\text{and}\quad\frac{\|A^m\|}{\|B^m\|}>M.\]
\end{proof}
\begin{example}
Consider, for example, $n=7$ and $m=13$. Then the construction above leads to taking
\[A_t:=C_t\oplus\begin{pmatrix}
0&75t^6\\0&0\end{pmatrix}\quad\text{and}\quad B_t:=C_t\oplus\begin{pmatrix}
0&0\\0&0\end{pmatrix}\]
with \[C_t=\begin{pmatrix}0 & t & 0 & {{t}^{3}} & -\frac{4426{{t}^{4}}}{161} & \frac{85976{{t}^{5}}}{483} & -\frac{109426{{t}^{6}}}{161} & \frac{336899{{t}^{7}}}{161}\\
0 & 0 & t & 0 & {{t}^{3}} & -\frac{4426{{t}^{4}}}{161} & \frac{85976{{t}^{5}}}{483} & -\frac{109426{{t}^{6}}}{161}\\
0 & 0 & 0 & t & 0 & {{t}^{3}} & -\frac{4426{{t}^{4}}}{161} & \frac{85976{{t}^{5}}}{483}\\
0 & 0 & 0 & 0 & t & 0 & {{t}^{3}} & -\frac{4426{{t}^{4}}}{161}\\
0 & 0 & 0 & 0 & 0 & t & 0 & {{t}^{3}}\\
0 & 0 & 0 & 0 & 0 & 0 & t & 0\\
0 & 0 & 0 & 0 & 0 & 0 & 0 & t\\
0 & 0 & 0 & 0 & 0 & 0 & 0 & 0\end{pmatrix}.\]
The table below gives values of $t$, calculated numerically with Scilab, such that 
\[\frac{\|(I+A_t)^7\|}{\|(I+B_t)^7\|}>M\quad\text{and}\quad\frac{\|(I+A_t)^{13}\|}{\|(I+B_t)^{13}\|}>M .\]

 \[  \text{
    \begin{tabular}{|c|c|c|c|c|c|c|c|}
    \hline
     Value of $M$   &  $10$   &  $10^2$   &  $10^3$  &  $10^4$  &  $10^5$   & $10^6$  & $10^7$   \\ \hline
     Value of $t$   &   $6$   &   $42$   &  $415$   &   $4142$  &   $41414$ & $414135$  & $4141340$ \\ \hline
    \end{tabular}}\]
To complete this example, we finish by giving the plot with the Matlab package Eigtool, of pseudospectra of $A=I+A_t$ and $B=I+B_t$, for $t=6$ and $t=42$.
\begin{figure}[h]
\caption{Pseudospectra of $A=I+A_t$ and $B=I+B_t$.}
\begin{subfigure}[b]{0.4\textwidth}
        \includegraphics[width=\textwidth]{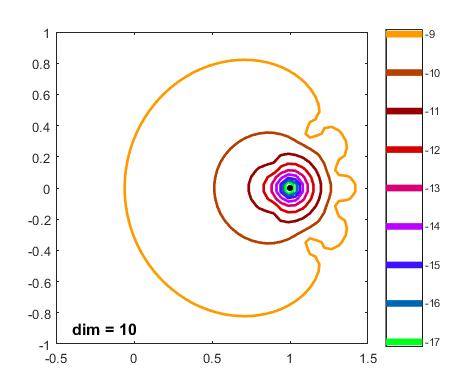}
        \caption*{t=6}
    \end{subfigure}
\begin{subfigure}[b]{0.4\textwidth}
       \includegraphics[width=\textwidth]{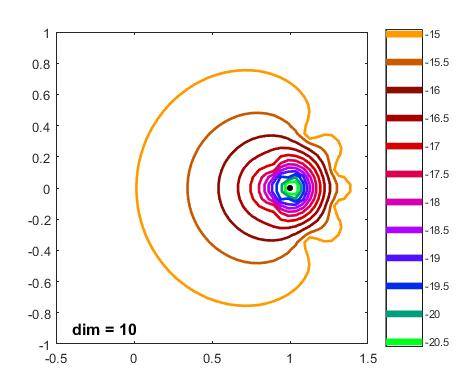}
        \caption*{t=42}
    \end{subfigure}
\end{figure}
\end{example}
\section{Case of two functions.}\label{SecF}
As in the proof of Theorem \ref{TheoP}, we will construct $C_t$ and find its coefficients to have $\|f(C_t)\|= O(t^5)$ and $\|g(C_t)\|=O(t^5)$ when $t\to\infty$. The next step will consist in verifying that the condition of Lemma \ref{Lemma} holds for this $C_t$. We will finish by the construction of $A_t$ and $B_t$ like in the proof of Theorem \ref{TheoP} and check that we can choose a sufficiently large $t$ such that these matrices give the result. 
\begin{proof}[Proof of Theorem \ref{TheoF}] 
Without loss of generality, we can suppose that $0\in\Omega$, $f'(0)\neq0$, $g'(0)\neq0$, $\tilde h_{1,2}(0)\neq0$ and $f$ and $g$ fail to satisfy the equation $(\star)$ for $z=0$.
For $k\in\mathbb N$, we define $f_k:=\frac1{k!}f^{(k)}(0)$ and $g_k:=\frac1{k!}g^{(k)}(0)$. And for $i,j\in\mathbb N$, we define $h_{i,j}:=f_ig_j-f_jg_i$.\\
Let $S$ be the unilateral shift on $\mathbb C^8$. Let $C_t$ be defined by
\[C_t:=tS+ut^3S^3+c_4t^4S^4+c_5t^5S^5+c_6t^6S^6+c_7t^7S^7\in M_8(\mathbb C).\]
$\bullet$ We begin by finding the $c_k$, rational functions in $u$, such that $\|f(C_t)\|= O(t^5)$ and $\|g(C_t)\|=O(t^5)$ when $t\to\infty$. Since $S^k=0$ for $k\ge8$, then $C_t^k=0$ if $k\ge8$. So $f(C_t)$ and $g(C_t)$ are determined by the Taylor polynomial of degree $7$ of $f$ and $g$. The Taylor developments of $f$ and $g$ are, when $z\to0$, %verifier l'orthographe
\[f(z)=f_0+f_1z+f_2z^2+f_3z^3+f_4z^4+f_5z^5+f_6z^6+f_7z^7+O(z^8)\]
\[g(z)=g_0+g_1z+g_2z^2+g_3z^3+g_4z^4+g_5z^5+g_6z^6+g_7z^7+O(z^8).\]

We have 
\begin{align*}
f(C_t)&=f_0I+f_1C_t+f_2C_t^2+f_3C_t^3+f_4C_t^4+f_5C_t^5+f_6C_t^6+f_7C_t^7\\
&=f_0I+f_1tS+f_2t^2S^2+(f_3+uf_1)t^3S^3\\
&\quad +(f_4+2uf_2+c_4f_1)t^4S^4+(f_5+3uf_3+2c_4f_2+c_5f_1)t^5S^5\\
&\quad +(f_6+4uf_4+3c_4f_3+(u^2+2c_5)f_2+c_6f_1)t^6S^6\\
&\quad +(f_7+5uf_5+4c_4f_4+3(u^2+c_5)f_3+2(uc_4+c_6)f_2+c_7f_1)t^7S^7
\intertext{and}
g(C_t)&=g_0I+g_1C_t+g_2C_t^2+g_3C_t^3+g_4C_t^4+g_5C_t^5+g_6C_t^6+g_7C_t^7\\
&=g_0I+g_1tS+g_2t^2S^2+(g_3+ug_1)t^3S^3\\
&\quad +(g_4+2ug_2+c_4g_1)t^4S^4+(g_5+3ug_3+2c_4g_2+c_5g_1)t^5S^5\\
&\quad +(g_6+4ug_4+3c_4g_3+(u^2+2c_5)g_2+c_6g_1)t^6S^6\\
&\quad +(g_7+5ug_5+4c_4g_4+3(u^2+c_5)g_3+2(uc_4+c_6)g_2+c_7g_1)t^7S^7.
\end{align*}
Thus the $c_k$ are the solutions of
\[ \left\{
\begin{matrix}
f_1c_7 & +2f_2c_6 & +3f_3c_5 & +(4f_4+2uf_2)c_4 & = & -(3f_3u^2+5f_5u+f_7) \\
       &  f_1c_6  & +2f_2c_5 &    +3f_3c_4      & = & -(f_2u^2+4f_4u+f_6)  \\
g_1c_7 & +2g_2c_6 & +3g_3c_5 & +(4g_4+2ug_2)c_4 & = & -(3g_3u^2+5g_5u+g_7) \\
       &  g_1c_6  & +2g_2c_5 &    +3g_3c_4      & = & -(g_2u^2+4g_4u+g_6).  \\
\end{matrix}
\right.\]
If we define
 \begin{align*}
U&= \det\begin{pmatrix}
f_1&2f_2&3f_3&4f_4+2uf_2\\
0&f_1&2f_2&3f_3\\
g_1&2g_2&3g_3&4g_4+2ug_2\\
0&g_1&2g_2&3g_3\\
\end{pmatrix},
\intertext{then}
U&=
(4f_1^2g_2^2-8f_1f_2g_1g_2+4f_2^2g_1^2)u\\
&\quad+f_1^2(8g_2g_4-9g_3^2)+f_1f_2(12g_2g_3-8g_1g_4)+f_3(f_1(18g_1g_3-12g_2^2)\\
&\quad +12f_2g_1g_2)-12f_2^2g_1g_3+f_4(8f_2g_1^2-8f_1g_1g_2)-9f_3^2g_1^2\\
&= 4h_{1,2}^2u+O(1).
\end{align*}
Therefore we can choose $u$ such that $U\neq0$ and then the system has a unique solution. Moreover the solutions $c_k$ satisfy, when $u\to\infty$,
\begin{align*}
b_4:=Uc_4=
&-((3f_1^2g_2g_3-3f_1f_2g_1g_3+f_3(3f_2g_1^2-3f_1g_1g_2))u^2+O(u)\\
b_5:=Uc_5=
&-((2f_1^2g_2^2-4f_1f_2g_1g_2+2f_2^2g_1^2)u^3+O(u^2)\\
b_6:=Uc_6=
&((f_1f_2(16g_2g_4-9g_3^2)-16f_2^2g_1g_4+f_3(9f_1g_2g_3+9f_2g_1g_3)\\
&\qquad +f_4(16f_2g_1g_2-16f_1g_2^2)-9f_3^2g_1g_2)u^2+O(u)\\
b_7:=Uc_7=&
((6f_1f_2g_2g_3-6f_2^2g_1g_3+f_3(6f_2g_1g_2-6f_1g_2^2))u^3+O(u^2).
\end{align*}
$\bullet$ Now, we  verify the conditions on the resolvent. When $u\to\infty$, we have
\begin{align*}
A':=U^2A&=2U^2u^3-Ub_5u+b_4^2\\
&=32(f_1g_2-f_2g_1)^4u^{5}+O(u^4)\\
&=32h_{1,2}^4u^{5}+O(u^4)\\
%B':=U^2B%&=3Ub_4u^2-Ub_6u+b_4b_5\\
%%&=-30(f_1g_2-f_2g_1)^3(f_1g_3-f_3g_1)u^5+O(u^{4})\\
%&=-30h_{1,2}^3h_{1,3}u^5+O(u^{4})\\
%C':=U^2C%&=U^2u^4+2Ub_5u^2+(2b_4^2-Ub_7)u+b_4b_6\\
%%&=-2(f_1g_2-f_2g_1)^2(32f_1^2g_2g_4 -32f_1f_2g_1g_4-27f_1^2g_3^2 +12f_1f_2g_2g_3+54f_1f_3g_1g_3-12f_2^2g_1g_3-12f_1f_3g_2^2\\&\qquad  -32f_1f_4g_1g_2+12f_2f_3g_1g_2+32f_2f_4g_1^2-27f_3^2g_1^2)u^5+O(u^4)\\
%&=-2h_{1,2}^2(32h_{1,2}h_{1,4}-h_{1,3}^2+12h_{1,2}h_{2,3})u^5+O(u^4)\\
%D':=U^2D%&=Ub_6u^2+b_4b_7\\
%%&=2(f_1g_2-f_2g_1)^2(32f_1f_2g_2g_4-32f_2^2g_1g_4-27f_1f_2g_3^2 +27f_1f_3g_2g_3+27f_2f_3g_1g_3-32f_1f_4g_2^2+32f_2f_4g_1g_2\\&\qquad  -27f_3^2g_1g_2)u^5+O(u^4)\\
%&=2h_{1,2}^2(32h_{1,2}h_{2,4}-27h_{1,3}h_{2,3})u^5+O(u^4)\\
E':=U^5E&=U^4(A^2(2Uu^2+b_5)-ACUu+B^2Uu-ABb_4)\\
&=a(f_1g_2-f_2g_1)^{10}u^{13}+O(u^{12})\\
&=ah_{1,2}^{10}u^{13}+O(u^{12})\qquad (\text{with }a=6144).\\
%F':=U^5F%&=U^4(A^2(2b_4u+b_6)+C(BUu-Ab_4)-ADUu)\\
%%&=-a(f_1g_2-f_2g_1)^9(f_1g_3-f_3g_1)u^{13}+O(u^{12})\\
%&=-ah_{1,2}^9h_{1,3}u^{13}+O(u^{12})\\
%G':=U^5G%&=U^4(D(BUu-Ab_4)+A^2b_7)\\
%%&=a(f_1g_2-f_2g_1)^9(f_3g_2-f_2g_3)u^{13}+O(u^{12})\\
%&=ah_{1,2}^9h_{2,3}u^{13}+O(u^{12})\\
%H':=U^{12}H%&=H^{12}(-AEG+AF^2-BEF+CE^2)\\
%%&=-b(f_1g_2-f_2g_1)^{22}((f_1g_2-f_2g_1)(f_1g_4-f_4g_1)-(f_1g_3-f_3g_1)^2+(f_1g_2-f_2g_1)(f_2g_3-f_3g_2))u^{31}+O(u^{30})\\
%&=bh_{1,2}^{22}(h_{1,4}h_{1,2}-h_{1,3}^2+h_{1,2}h_{2,3})u^{31}+O(u^{30})\\
%I':=U^{12}I%&=U^{12}((AF-BE)G+DE^2)\\
%%&=b(f_1g_2-f_2g_1)^{22}((f_1g_2-f_2g_1)(f_2g_4-f_4g_2)-(f_1g_2-f_2g_1)(f_2g_3-f_3g_2))u^{31}+O(u^{30})\\
%&=bh_{1,2}^{22}(h_{2,4}h_{1,2}-h_{1,3}h_{2,3})u^{31}+O(u^{30})   \qquad(\text{with } b=301989888)\\
J':=U^{29}J&=U^{29}(EI^2-FHI+GH^2)\\
&=ab^2h_{1,2}^{55}[h_{1,2}(h_{2,4}^2-h_{3,4}(2h_{2,3}+h_{1,4}))+h_{1,3}(h_{1,3}h_{3,4}\\
&\qquad-h_{2,3}h_{2,4})+h_{2,3}^3]u^{75}+O(u^{74}) \qquad(\text{with } b=301989888). 
\end{align*}
Thus we can finally choose $u\neq0$ such that we also have $A\neq0$, $E\neq0$ and $J\neq0$.
By Lemma \ref{Lemma}, there exists $\mu>0$ such that 
\[\|(I-zC_t)^{-1}\|\ge 1+\mu t^6|z|\qquad(t\ge1,z\in\mathbb C).\]
$\bullet$ We finish the proof with the construction of $A$ and $B$.\\
We define $A_t:= C_t\oplus\begin{pmatrix}
0&\mu t^6\\0&0
\end{pmatrix}\in M_{10}(\mathbb C)$ and $B_t:= C_t\oplus\ O_2\in M_{10}(\mathbb C)$.
We have, for all $t\ge 1$ and $z\in\mathbb C$,
\[\left\|\left(I-z\begin{pmatrix}
0&\mu t^6\\0&0
\end{pmatrix}\right)^{-1}\right\| =\left\|\begin{pmatrix}
1&z\mu t^6\\0&1
\end{pmatrix}\right\|\le 1+\mu t^6|z|.\]
Then 
\[\|(I-zA_t)^{-1}\|=\|(I-zC_t)^{-1}\|=\|(I-zB_t)^{-1}\|.\]
Now, if $t\ge1$, we see that 
\begin{align*}
\|f(A_t)\|&\ge\left\|f\left(\begin{pmatrix}
0&\mu t^6\\0&0
\end{pmatrix}\right)\right\|=\left\| \begin{pmatrix}
f_0&f_1\mu t^6\\0&f_0
\end{pmatrix}\right\|\ge |f_1|\mu t^6 \\
\|f(B_t)\|&=\max(\|f(C_t)\|, |f(0)|)=\|f(C_t)\|=O(t^5)\quad(t\to\infty)\\
\|g(A_t)\|&\ge\left\|g\left(\begin{pmatrix}
0&\mu t^6\\0&0
\end{pmatrix}\right)\right\|=\left\|\begin{pmatrix}
g_0&g_1\mu t^6\\0&g_0
\end{pmatrix}\right\|\ge |g_1|\mu t^6 \\
\|g(B_t)\|&=\max(\|g(C_t)\|, |g(0)|)=\|g(C_t)\|=O(t^5)\quad(t\to\infty).
\end{align*}
Since $ |f_1|\mu>0$ and $|g_1|\mu >0$, we have
\[\lim_{t\to\infty}\frac{\|f(A_t)\|}{\|f(B_t)\|}=+\infty\quad\text{and}\quad \lim_{t\to\infty}\frac{\|g(A_t)\|}{\|g(B_t)\|}=+\infty.\]
So, finally, we can choose $t\ge1$ such that
\[\frac{\|f(A_t)\|}{\|f(B_t)\|}>M\quad\text{and}\quad \frac{\|g(A_t)\|}{\|g(B_t)\|}>M.\]
\end{proof}
\section{Final remarks and questions}

$\bullet$ The condition $(\star)$ is quite mysterious and one might wonder which functions satisfy this condition. In the spirit of \cite{RansRaoua2013}, Theorem 1.2, we shall show that the condition $(\star)$ precludes $f$ and $g$ from being M\"obius transformations. More precisely, we show that, if $\lambda f+\mu g$ is a M\"obius transformation for some $\lambda,\mu\in\mathbb C$, then $f$ and $g$ do not satisfy $(\star)$.\\
 Indeed, without loss of generality, by replacing $f$ or $g$ by $\lambda f+\mu g$, we can assume that $f$ or $g$ is a M\"obius transformation. We have:
\begin{align*}
\tilde h_{1,2}(&\tilde h_{2,4}^2-\tilde h_{3,4}(2\tilde h_{2,3}+\tilde h_{1,4}))+\tilde h_{1,3}(\tilde h_{1,3}\tilde h_{3,4}-\tilde h_{2,3}\tilde h_{2,4})+\tilde h_{2,3}^3\\
&=\frac1{12}(\tilde h_{1,4}-\tilde h_{2,3})(Sf.G+F.Sg)-\frac{\tilde h_{2,4}}{576}(Sf.(Sg)'+(Sf)'.Sg)\\
&-\frac{\tilde h_{1,3}}{48}((Sf)'.G+F.(Sg)')+\frac{\tilde h_{3,4}}{72}Sf.Sg+\frac{\tilde h_{2,3}}{1152}(Sf)'.(Sg)'+2\tilde h_{1,2}F.G
\end{align*}
with $F=\frac1{48}f'' f^{(4)}-\frac1{36}(f^{(3)})^2$ and $Sf:=2(f')^2(\text{S}f)=2f' f^{(3)}-3(f'')^2$ (with S$f$, the Schwarzian derivative of $f$) and likewise for $G$ and $Sg$.\\ %Reformuler car dériv Schwarz pas égale a ça exactement.
If $f$ is a M\"obius transformation, then its Schwarzian derivative is zero. So we have  $Sf=0$ and then $(Sf)'=0$. Moreover, $f$ is a M\"obius transformation if and only if the function $f_w$, defined by $f_w(z):=\frac{f(z)-f(w)}{z-w}$, is a M\"obius transformation, for all $w$ in the domain of the function $f$. Then $Sf_w=0$, and we deduce that $F(w)=\frac1{12}Sf_w(w)=0$. The same applies if $g$ is a M\"obius transformation.\\
Since some other functions do not satisfy the condition $(\star)$, this raises the question of whether $(\star)$ has an interpretation analogous to the condition ``$f$ is not a M\"obius transformation'' in \cite{RansRaoua2013}.\\
\\
$\bullet$ The condition ``$f$ and $g$ satisfy $(\star)$'' is not necessary, in Theorem \ref{TheoF}. Indeed, Section \ref{SecP} provides plenty of counterexamples (for example $f(z)=z^2$ and $g(z)=z^7$). What might a necessary and sufficient condition look like?

\section*{Acknowledgments}
I would like to thank Thomas Ransford for his  helpful
remarks and support. I was supported by a FRQNT doctoral scholarship.

\bibliographystyle{plain}
\bibliography{biblio}

\begin{thebibliography}{1}

\bibitem{GreenTref1993}
Anne Greenbaum and Lloyd~N. Trefethen.
\newblock Do the pseudospectra of a matrix determine its behavior?
\newblock Technical Report TR 93-1371, Cornell University, Ithaca, NY, 1993.

\bibitem{Kreiss1962}
Heinz-Otto Kreiss.
\newblock \"{U}ber die {S}tabilit\"{a}tsdefinition f\"{u}r
  {D}ifferenzengleichungen die partielle {D}ifferentialgleichungen
  approximieren.
\newblock {\em Nordisk Tidskr. Informationsbehandling (BIT)}, 2:153--181, 1962.

\bibitem{LevTref1984}
Randall~J. LeVeque and Lloyd~N. Trefethen.
\newblock On the resolvent condition in the {K}reiss matrix theorem.
\newblock {\em BIT}, 24(4):584--591, 1984.

\bibitem{Rans2007}
Thomas Ransford.
\newblock On pseudospectra and power growth.
\newblock {\em SIAM J. Matrix Anal. Appl.}, 29(3):699--711, 2007.

\bibitem{Rans2010}
Thomas Ransford.
\newblock Pseudospectra and matrix behaviour.
\newblock In {\em Banach algebras 2009}, volume~91 of {\em Banach Center
  Publ.}, pages 327--338. Polish Acad. Sci. Inst. Math., Warsaw, 2010.

\bibitem{RansRaoua2013}
Thomas Ransford and Samir Raouafi.
\newblock Pseudospectra and holomorphic functions of matrices.
\newblock {\em Bull. Lond. Math. Soc.}, 45(4):693--699, 2013.

\bibitem{TrefEmbr2005}
Lloyd~N. Trefethen and Mark Embree.
\newblock {\em Spectra and pseudospectra}.
\newblock Princeton University Press, Princeton, NJ, 2005.
\newblock The behavior of nonnormal matrices and operators.

\bibitem{WegTref1994}
Elias Wegert and Lloyd~N. Trefethen.
\newblock From the {B}uffon needle problem to the {K}reiss matrix theorem.
\newblock {\em Amer. Math. Monthly}, 101(2):132--139, 1994.

\end{thebibliography}
\nocite{*}
\end{document}